\documentclass[12pt,a4paper]{amsart}
\usepackage{amstext}
\usepackage{amsthm}
\usepackage{amsopn}
\usepackage{amsfonts}
\usepackage{amsmath}
\usepackage{amssymb}
\usepackage{amscd}

\newcommand{\all}[2]{%
\left\{\, {#1}\vphantom{{#1}^k{#2}^k}\,\right |%
\left. \vphantom{{#1}^k{#2}^k} {#2}\, \right\}}

\newcommand{\norm}[1]{\left\Vert#1\right\Vert}
\newcommand{\abs}[1]{\left\vert#1\right\vert}
\newcommand{\set}[1]{\left\{#1\right\}}

\newcommand{\N}{\ensuremath{\mathbb N}}
\newcommand{\R}{\ensuremath{\mathbb R}}

\newcommand{\C}{\ensuremath{\mathbb C}}

\newcommand{\lam}{\lambda}

\theoremstyle{plain}
\newtheorem{theorem}{Theorem}[section]
\newtheorem{prop}[theorem]{Proposition}
\newtheorem{lem}[theorem]{Lemma}
\newtheorem{coro}[theorem]{Corollary}

\theoremstyle{definition}
\newtheorem{defn}[theorem]{Definition}
\newtheorem{rem}[theorem]{Remark}
\newtheorem{ex}[theorem]{Example}
\newtheorem{notation}[theorem]{Notation}

\newcommand{\cal}{\mathcal}

\newcommand{\Si}{{\cal S}_\infty}

\newcommand{\cc}{\subset\!\!\!\subset}

\begin{document}

\title[Eigenvalue decay]
{Eigenvalue decay of
operators on harmonic function spaces}
\author[O.F. ~Bandtlow and C.-H.~Chu]
{Oscar F.~Bandtlow and Cho-Ho Chu}

\address{Oscar F.~Bandtlow,
School of Mathematical Sciences, Queen Mary, University
of London, London E1 4NS, UK}

\address{Cho-Ho Chu, 
School of Mathematical Sciences, Queen Mary, University
of London, London E1 4NS, UK}

\date{February 26, 2009}


\begin{abstract}
Let $\Omega$ be an open set in $\R^d$ $(d > 1)$ and $h(\Omega)$ the
Fr\'{e}chet space of harmonic functions on  $\Omega$. Given a
bounded linear operator \mbox{$L :h(\Omega)\to h(\Omega)$},
we show that its eigenvalues $\lambda_n$, arranged in decreasing order
and 
counting multiplicities,
satisfy $|\lambda_n|\leq K\exp(-cn^{1/(d-1)})$, where $K$ and $c$ are
two explicitly
computable positive constants.
\end{abstract}

\maketitle

\section{Introduction}

In his celebrated memoirs \cite{Gro}, Grothendieck showed that the
eigenvalues, 
ordered by magnitude and counting algebraic multiplicities, 
of every bounded operator on a
quasi-complete nuclear space decrease {\it rapidly} \cite[Chap II,
\S 2, No.~4, Corollaire 3]{Gro}. He also remarked in \cite[Chap
II, \S 2, No.~4, Remarque 9]{Gro} that this result could be
improved for certain spaces including the space ${\cal H}(\Omega)$
of holomorphic functions on a domain $\Omega$ in $\C^d$. Indeed,
he showed that the eigenvalues $$|\lambda_1| \geq |\lambda_2| \geq
\cdots \geq |\lambda_n| \geq \cdots$$ of a bounded operator on
${\cal H}(\Omega)$ satisfy 
\begin{equation} 
\label{decayest}
\lambda_n = O(\exp (-cn^{1/d}))
\end{equation}
for
some positive constant $c$. We note that Grothendieck originally
asserted that $\lam_n =O(\exp(-c n))$ although his arguments
actually yield the above dimension-dependent decay (see \cite[Appendix
A]{BJ} for a short alternative proof of (\ref{decayest})). 

In this paper, we show that a decay of this type also occurs for the
eigenvalues of bounded operators on the space $h(\Omega)$ of
harmonic functions on a non-empty open set $\Omega$ in $\R^d$ $(d >
1)$. More precisely, we shall show that there are positive constants
$K$ and $c$, such that
\begin{equation}
\label{eq:decay}
|\lambda_n|\leq K \exp(-cn^{1/(d-1)})\,.
\end{equation}
The main ingredient in our proof is to show, by using properties of
spherical harmonics, that the singular values of certain natural
embeddings of harmonic Bergman spaces on balls decay at a stretched
exponential rate (see Proposition~\ref{berg:prop5}, which, in fact,
yields the precise asymptotics of the singular value decay of these
embeddings). After extending this result to embeddings associated
with more general open sets (see Theorem~\ref{embeddingtheorem})
some abstract operator theoretic techniques discussed in
Section~\ref{expoclasses} then yield the main result in
Theorem~\ref{maintheorem}. This method also yields explicitly
computable estimates for the constants $K$ and $c$ occurring in
(\ref{eq:decay}), which will be illustrated by a concrete example at
the end of this article. 

\begin{notation} Let $H_1$ and $H_2$ be Hilbert spaces.
Throughout, we use ${\cal L}(H_1,H_2)$ to denote the Banach space
of bounded linear operators from $H_1$ to $H_2$, equipped with the
usual norm, and $\Si(H_1,H_2)\subset {\cal L}(H_1,H_2)$ to denote
the closed subspace of compact operators from $H_1$ to $H_2$. We
shall often write $\cal L$ or $\Si$ if the Hilbert spaces $H_1$
and $H_2$ are understood.

For $A\in\Si(H,H)$ we let
$\lambda(A)=\set{\lam_n(A)}_{n=1}^\infty$ denote the sequence of
eigenvalues of $A$, each eigenvalue repeated according to its
algebraic multiplicity, and ordered by magnitude, so that
$\abs{\lam_1(A)}\geq\abs{\lam_2(A)}\geq\ldots$. We also write
$|\lambda(A)|$ for the sequence $\set{|\lam_n(A)|}_{n=1}^\infty$.

Similarly, for $A\in\Si(H_1,H_2)$, we use
$s(A)=\set{s_n(A)}_{n=1}^\infty$, where
\[s_n(A)= \sqrt{\lambda_n(A^*A)}\quad (n\in\N)\,,\]
to denote the
sequence of \textit{singular values} of $A$.
\end{notation}

\section{Exponential classes}
\label{expoclasses}
In this section we consider classes of compact operators whose singular values
decay exponentially of a particular order and list some of their
properties. We start by introducing
exponential classes of complex sequences. Let $a>0$ and
$\alpha>0$. We define $$  {\cal E} (a,\alpha):= \all{x\in \C^\N}{
\sup_{n\in\N}
  |x_n| \exp(an^{\alpha}) < \infty}. $$
Then ${\cal E}(a,\alpha)$ is a complex Banach space  with norm
$$|x|_{a,\alpha}:= \sup_{n\in\N}
  |x_n| \exp(an^{\alpha})\,,$$
which we call {\it exponential class of sequences of type}
$(a,\alpha)$. The set
\[ {\cal E} (\alpha):=\bigcup_{a>0}{\cal E}(a,\alpha) \]
will be referred to as {\it exponential class of sequences of type}
$\alpha$.

\begin{defn} Let $H_1$ and $H_2$ be Hilbert spaces, which, to avoid
  trivialities, we assume to be infinite dimensional.
For $a, \alpha > 0$, we define

\[ E(a,\alpha;H_1,H_2):=\all{A\in\Si(H_1,H_2)}{ |A|_{a,\alpha}:=\sup_{n\in\N}
  s_n(A)\exp(an^{\alpha}) < \infty}\,, \]
which is called {\it  exponential class of operators of type}
$(a, \alpha)$. We refer to the number $|A|_{a,\alpha}$ as {\it
$(a,\alpha)$-gauge} or simply {\it
  gauge} of $A$. The set
\[ E(\alpha;H_1,H_2):=\bigcup_{a >0}E(a,\alpha;H_1,H_2) \]
is called the {\it  exponential class of operators of type} $
\alpha$. It consists of compact operators between $H_1$ and $H_2$
whose singular values  decay at a stretched exponential
rate with stretching exponent $\alpha$.
\end{defn}
Whenever the Hilbert spaces are clear from the context, we
suppress reference to them and simply write $E(a,\alpha)$ instead
of $E(a,\alpha;H_1,H_2)$ and similarly for $E(\alpha)$.

We now collect a number of properties of the classes $E(a,\alpha)$
which will be used later.

\begin{prop} \label{prop2}
Let  $\alpha, a, a_1,\ldots, a_N >0$.
  \begin{itemize}
  \item[(i)] If $A ,C\in \cal L$ and $B\in E(a,\alpha)$, then $$ABC
  \in E(a, \alpha) \quad and \quad
 |ABC|_{a,\alpha}\leq \norm{A}{}\, |B|_{a,\alpha}\, \norm{C}{}.
$$
\item[(ii)] Let $A_n\in E(a_n,\alpha)$ for $1\leq n\leq
  N$ and let
  $A =\sum_{n=1}^N A_n$. Then
\[ A \in E(a',\alpha)\text{ with } |A|_{a',\alpha}\leq
  N \max_{1\leq n\leq N}|A_n|_{a_n,\alpha} \]
  where $a':=(\sum_{n=1}^Ka_n^{-1/\alpha})^{-\alpha}$.
In particular
\[ E(a_1,\alpha)+\cdots + E(a_N,\alpha)\subset E(a',\alpha) \]
and the above inclusion is sharp in the sense that
\[ E(a_1,\alpha)+\cdots + E(a_N,\alpha)\not\subset E(b,\alpha) \]
for $b>a'$
\item[(iii)] If $A\in E(a,\alpha)$, then
\[ \lam(A)\in {\cal E}(a/(1+\alpha),\alpha) \text{ with }
|\lam(A)|_{a/(1+\alpha),\alpha}\leq |A|_{a,\alpha}. \]
The result is sharp in the sense that there is an operator $A\in
E(a,\alpha)$ such that $\lam(A)\not\in {\cal E}(b,\alpha)$ whenever
$b>a/(1+\alpha)$.
\end{itemize}
\end{prop}
\begin{proof}
 See \cite[Prop.~2.5, 2.8, and 2.10]{Ban}.
\end{proof}
\begin{rem}\label{rem2}
Note that $E(a,\alpha;H_1,H_2)$ is not a linear space. In order to see
this assume for simplicity that $H_1=H_2=\ell^2$. Let 
$\sigma_n=\exp(-an^\alpha)$ and let $A$ and $B$ be the diagonal
operators $A=diag(\sigma_1,0,\sigma_2,0,\ldots)$ and 
$B=diag(0,\sigma_1,0,\sigma_2,\ldots)$. Then $A,B\in E(a,\alpha)$ with 
$|A|_{a,\alpha}= |B|_{a,\alpha}=1$ but 
$A+B=diag(\sigma_1,\sigma_1,\sigma_2,\sigma_2,\ldots)\not \in
E(a,\alpha)$. Essentially the same construction can be used to deal
with the case of 
arbitrary (infinite-dimensional) spaces $H_1$ and $H_2$. 

The fact that $E(a,\alpha)$ is not a linear space also follows from 
assertion (ii) of the previous proposition,  
which implies that $E(a,\alpha)+E(a,\alpha)\subset
 E(2^{-\alpha}a,\alpha)$, but $E(a,\alpha)+E(a,\alpha)\not\subset
 E(a,\alpha)$, because $2^{-\alpha}a<a$.
\end{rem}

\section{Harmonic Bergman spaces and canonical identifications}
\label{berg}
To pave the way for the main result, our objective in
this section will be to show that certain natural embeddings of harmonic
Bergman spaces have singular values which decay at a stretched
exponential rate.

In the sequel, all open sets in $\R^d$ are non-empty. Let
$\Omega\subset \R^d$ be an open set and let $L^2(\Omega)$ be the
Lebesgue space of complex-valued, square-integrable functions on
$\Omega$ with
respect to Lebesgue measure $dx$ on $\R^d$, equipped with the usual
norm.
Let $\Delta$ be the Laplace operator and
let
\[ h^2(\Omega):=\all{ f \in L^2(\Omega)}%
{ \Delta f = 0} \] be the {\it  harmonic Bergman space} over
$\Omega$, which  is a separable Hilbert space with inner product
\[ (f,g)_{h^2(\Omega)}=\int_\Omega f(x) \overline{g(x)} \, dx \quad
(f,g\in h^2(\Omega))\,. \] We refer to \cite{ABR} for more details
about $h^2(\Omega)$.

Suppose that $\Omega_1 , \Omega_2\subset \R^d$ are open and that
$\Omega_2\subset \Omega_1$. By restriction to $\Omega_2$ every
element in $h^2(\Omega_1)$ can also be considered as an element of
$h^2(\Omega_2)$. This restriction yields a linear transformation
$J: h^2(\Omega_1)\to h^2(\Omega_2)$, called the
\textit{canonical identification}. If $\Omega_1$ is connected,
then the canonical identification is injective and hence a proper
embedding of $h^2(\Omega_1)$ in $h^2(\Omega_2)$.

It is obvious that $J$ is continuous.
Under stronger assumptions about $\Omega_1$ and $\Omega_2$,
more can be said about $J$. We require the following definition.

\begin{defn}
Let $\Omega_1, \Omega_2$ be open subsets of $\R^d$. If
$\overline{\Omega}_2$ is a compact subset of $\Omega_1$ then we
say that $\Omega_2$ is \emph{compactly contained} in $\Omega_1$,
and write $\Omega_2 \cc \Omega_1$.
\end{defn}

It turns out that if $\Omega_2\cc\Omega_1$ then
$J:h^2(\Omega_1)\hookrightarrow h^2(\Omega_2)$ is a
compact operator.\footnote{To see this, note that
$J(h^2(\Omega_1))$ is contained in the
Banach space $C^b(\Omega_2)$ of bounded continuous functions on
$\Omega_2$ and $J:h^2(\Omega_1) \to C^b(\Omega_2)$ has
closed graph. Hence $\{Jf: \|f\|_2 \leq 1\}$ is uniformly bounded
on $\Omega_2$ and therefore a normal family in $h^2(\Omega_2)$
(cf. \cite[Theorem 2.6]{ABR}).} In fact rather more is true: $J\in
E(1/(d-1))$. The proof of this result requires a certain amount of
preparation and will be presented in the next section.

In this section we shall be content with proving this result for
the case where $\Omega_1$ and $\Omega_2$ are concentric balls, in
which case the rate of decay can be identified precisely.

\begin{notation}
We denote by $B_{r,x}$ the ball with radius $r$ centred at $x$,
with respect to the Euclidean metric. We use $B_r$ as a short-hand
for  $B_{r,0}$. Given  a ball $B =B_{r,x}$ and $\gamma>0$, we use
the symbol
\[ B(\gamma):=B_{\gamma r,x} \]
for the $\gamma$-dilation of the ball.
\end{notation}

As usual, given $x= (x_1, \ldots, x_d) \in \R^d$ and a multi-index
$\alpha = (\alpha_1, \ldots, \alpha_d)$, we define $$x^\alpha =
x_1^{\alpha_1} \cdots x_d^{\alpha_d} \quad {\rm and}\quad |\alpha|
= \alpha_1 + \cdots +\alpha_d. $$ A {\it homogeneous harmonic
polynomial} of degree $k$ in $d$ dimensions is a polynomial $p :
\R^d \to \C$ of the form $$p(x) = \sum_{|\alpha|=k}
c_\alpha x^\alpha \qquad (x \in \R^d)\,,$$ which is also a harmonic
function. The restriction of $p$ to the unit sphere $S$ in $\R^d$
is called a {\it spherical harmonic} of degree $k$. The number
$N_d(k)$ of linearly independent spherical harmonics of degree $k$
in $d$ dimensions is given by the power series
$$\frac{1+x}{(1-x)^{d-1}} = \sum_{k=0}^\infty N_d(k)x^k \qquad
(|x| < 1)$$ (see \cite[Lemma 3]{m}). Hence
\begin{equation*}
N_d(k) =
{k+d-1\choose d-1}-{k+d-3 \choose d-1}
\end{equation*}
(see also \cite[Proposition 5.17]{ABR}). We now define $h_d(k) =
N_d(0) + \cdots + N_d(k)$, which is the number of linearly
independent homogeneous harmonic polynomials of degree {\it at
most} $k$ in $d$ dimensions. It follows that
\[ h_d(k) = \sum_{l=0}^k\left ({l+d-1\choose
d-1}-{l+d-3\choose d-1}\right ) = {k+d \choose d}-{k+d-2\choose
d}.\]
We note that $h_d(0) = 1$ and define $h_d(-1) = 0$.

\begin{prop}\label{berg:prop5}
 Let $B\subset \R^d$ be a ball and $\gamma>1$. Then
the singular values of the canonical identification
\[ J:h^2(B(\gamma))\hookrightarrow h^2(B) \]
are given by
\[ s_n(J)=\gamma^{-(k+\frac{d}{2})}\]
for $h_d(k-1) < n \leq h_d(k)$ and $k\in\N \cup\{0\}$.
\end{prop}
\begin{proof}
By translation invariance of the Lebesgue measure, we may assume
that $B$ is centred at the origin, say,  $B=B_r$. The proof relies
on the fact that $h^2(B_{r})$ and $h^2(B_{\gamma r})$ have a
common complete orthogonal system consisting of homogeneous
harmonic polynomials. To see this we first note that the linear
span of homogeneous harmonic polynomials is dense in $h^2(B_s)$
for every $s>0$ \cite[Lemma 8.8]{ABR}. Let now $f$ and $g$ be
homogeneous harmonic polynomials of degree $n$ and $m$
respectively. Then by the polar-coordinate formula for integration
in $\R^d$ (cf. \cite[p. 150]{rudin}), we have
\begin{align*}
(f,g)_{h^2(B_{r})} &=d\,
{\rm Vol}(B_1)\int_0^{r}\rho^{d-1}\int_Sf(\rho\xi)\overline{g(\rho\xi)}\,d\sigma(\xi)\,d\rho\\
&= d\, {\rm
Vol}(B_1)\int_0^{r}\rho^{d-1+n+m}\int_Sf(\xi)\overline{g(\xi)}\,d\sigma(\xi)\,d\rho
\end{align*}
where $S$ is the unit sphere in $\R^d$ and $\sigma$ the normalised
surface measure on $S$. Therefore
\begin{equation}\label{berg:e:calc}
(f,g)_{h^2(B_{r})}=d
{\rm Vol}(B_1)
\frac{r^{d+n+m}}{d+n+m}\int_Sf(\xi)\overline{g(\xi)}\,d\sigma(\xi).
\end{equation}
Since $\int_Sf\overline{g}\,d\sigma=0$ whenever $n\neq m$
\cite[Theorem 5.3]{ABR}, the Gram-Schmidt orthogonalisation
process yields an orthonormal basis for $h^2(B_{r})$, consisting
of homogeneous harmonic polynomials. Observe now that by
(\ref{berg:e:calc}) we have
\begin{equation}\label{berg:e:calc2}
(f,g)_{h^2(B_{\gamma r})}=\gamma^{d+n+m}(f,g)_{h^2(B_{r})}
\end{equation}
for any two homogeneous harmonic polynomials $f$ and $g$ of degree
$n$ an $m$ respectively. In particular $(f,g)_{h^2(B_{\gamma
r})}=0$ whenever $(f,g)_{h^2(B_{r})}=0$. This shows that
$h^2(B_{r})$ and $h^2(B_{\gamma r})$ have a common complete
orthogonal system consisting of homogeneous harmonic polynomials.

In order to see that the canonical identification $J:h^2(B_{\gamma
r})\hookrightarrow h^2(B_{r})$ has the desired properties, note
that by (\ref{berg:e:calc2}) we have
\[
(J^*Jf,g)_{h^2(B_{\gamma
    r})}=(Jf,Jg)_{h^2(B_{r})}=\gamma^{-(d+n+m)}(f,g)_{h^2(B_{\gamma r})}.
\]
This implies that $J^*J$ is diagonal
   with respect to the orthonormal basis of homogeneous harmonic polynomials.
   Its eigenvalues therefore belong to the set
   $\all{\gamma ^{-(2k+d)}}{k\in\N\cup \{0\}}$. Consequently, the singular values
   of $J$ belong to the set
   $\all{\gamma^{-(k+\frac{d}{2})}}{k\in\N \cup \{0\}}$.
 As there are $N_d(k)$
   linearly independent homogeneous harmonic polynomials of degree $k$, the
   value $\gamma ^{-(k+\frac{d}{2})}$ occurs with multiplicity
$N_d(k)$. If we order the orthonormal basis by degrees, then we
have $$s_n(J)=\gamma ^{-(k+\frac{d}{2})}$$ for $h_d(k-1) < n \leq
h_d(k)$.
\end{proof}
In order to study the singular value asymptotics of the canonical
identification $J$ obtained in the previous proposition, we require the
following lemma.
\begin{lem} \label{berg:lem:techlem}
Let $d\in\N$ and let $a_1, \ldots, a_d \geq 0$. Then
\begin{equation}
\label{berg:techlem}
\sup_{x\geq 0}
\prod_{k=1}^d(x+a_k)^{1/d}-x=
\lim_{x\rightarrow\infty}\prod_{k=1}^d(x+a_k)^{1/d}-x=\frac{1}{d}\sum_{k=1}^da_k.
\end{equation}
\begin{proof}
The case $d=1$ of (\ref{berg:techlem}) is clearly true, so suppose
$d\geq 2$.  Define $$h(x)=\prod_{k=1}^d(x+a_k)^{1/d}-x \qquad (x
\geq 0).$$ Then $h$ is an increasing function. Indeed, we have
\[
h'(x)=\frac{1}{d}\left (\prod_{k=1}^d(x+a_k)^{-1 + 1/d}\right
)\sum_{k=1}^d\prod_{\substack{l=1
    \\ l\neq k}}^d(x+a_l)-1 \]
where
\[ \frac{1}{d}\sum_{k=1}^d\prod_{\substack{l=1 \\ l\neq k}}^d(x+a_l)
\geq \prod_{k=1}^d(x+a_k)^{1-1/d}\,, \] since
\[ \frac{1}{d}\sum_{k=1}^d(x+a_k)^{-1}\geq \prod_{k=1}^d(x+a_k)^{-1/d} \]
by the arithmetic-geometric mean inequality. Thus $h'(x)\geq 0$ for
$x\geq 0$.

To complete the proof of (\ref{berg:techlem}), observe that
\[ \lim_{x\rightarrow \infty}h(x)=
\lim_{t \downarrow 0}t^{-1}\left ( \prod_{k=1}^{d}(1+a_kt)^{1/d}-1
\right )= \frac{1}{d}\sum_{k=1}^{d}a_k \] by
l'H\^{o}pital's rule.
\end{proof}
\end{lem}

\begin{prop} \label{ballasymp}
 Let $d\geq 2$ and $B\subset \R^d$ be a ball. Given $\gamma>1$,
the canonical embedding
\[ J:h^2(B(\gamma))\hookrightarrow h^2(B) \]
satisfies
\begin{equation}
\label{berg:prop6:e3}
J\in E(c,1/(d-1)),\quad\text{where }
c=\left (\frac{(d-1)!}{2}\right)^{1/(d-1)} \log \gamma
\end{equation}
and
\begin{equation}
|J|_{c,1/(d-1)}=\gamma^{-1/2}. \label{berg:prop6:e3i}
\end{equation}
In other words, its singular value sequence $s(J)$ has the
following asymptotics:
\begin{equation}
\lim_{n\rightarrow\infty}\frac{\log\abs{\log s_n(J)}}{\log
n}=\frac{1}{d-1}\,; \label{berg:prop6:e1}
\end{equation}

\begin{equation}
\lim_{n\rightarrow\infty} \frac{ \log
s_n(J)}{n^{1/(d-1)}}= - \left(\frac{(d-1)!}{2}\right)^{1/(d-1)}\log
\gamma;
\label{berg:prop6:e2}
\end{equation}

\begin{equation}\label{berg:prop6:e4}
\sup_{n\in\N} \left ( \log s_n(J) + \left
(n\frac{(d-1)!}{2}\right )^{1/(d-1)}\log \gamma\right
)=-\frac{1}{2}\log \gamma.
\end{equation}
\end{prop}
\begin{proof}
We have
\begin{equation}\label{berg:prop6:hd}
h_d(k)=\frac{2}{(d-1)!}(k+\frac{d-1}{2})\prod_{l=1}^{d-2}(k+l),
\end{equation}
where the product is interpreted to be equal to 1 if the upper range
  is strictly less than 1.
   By Proposition  \ref{berg:prop5}, we have
\begin{equation}
\frac{ \log \abs{\log \gamma^{-1}} + \log (k+\frac{d}{2})}{\log h_d(k)}\leq
\frac{ \log \abs{\log s_n(J)}}{\log n}\leq
\frac{ \log \abs{\log \gamma^{-1}} +
\log (k+\frac{d}{2})}{\log h_d(k-1)}\,. \label{berg:prop6:pe1}
\end{equation}

Using (\ref{berg:prop6:hd}), we obtain
\begin{equation}
\lim_{k\rightarrow \infty}\frac{\log (k+\frac{d}{2})}{\log h_d(k)}
=\lim_{k\to\infty}\frac{\log
(k+\frac{d}{2})}{\log h_d(k-1)}=\frac{1}{d-1}\,. \label{berg:prop6:pe2}
\end{equation}
Combining (\ref{berg:prop6:pe1}) and  (\ref{berg:prop6:pe2}), the
 assertion  (\ref{berg:prop6:e1}) follows.

Similarly, we have
\[
\frac{(k+\frac{d}{2})|\log \gamma^{-1}|}{h_d(k)^{1/(d-1)}} \leq
\frac{|\log s_n(J)|}{n^{1/(d-1)}} \leq \frac{(k+\frac{d}{2})|\log
\gamma^{-1}|}{h_d(k-1)^{1/(d-1)}}.
\]
Since
\[
\lim_{k\rightarrow
\infty}\frac{(k+\frac{d}{2})}{h_d(k)^{1/(d-1)}}=\lim_{k\to\infty}\frac{
(k+\frac{d}{2})}{h_d(k-1)^{1/(d-1)}}=\left(\frac{(d-1)!}{2}\right)^{1/(d-1)},
\]
equation (\ref{berg:prop6:e2}) follows.

It remains to establish (\ref{berg:prop6:e4}). By
Lemma~\ref{berg:lem:techlem}, we have
\begin{align*}
  \log s_n(J)+\left(n\frac{(d-1)!}{2}\right)^{1/(d-1)}\log
  \gamma &=
  \left ( \left(n\frac{(d-1)!}{2}\right )^{1/(d-1)}-(k+\frac{d}{2}) \right
  )\log \gamma\\
  &\leq  \left ( \left(h_d(k)\frac{(d-1)!}{2}\right)^{1/(d-1)}-(k+\frac{d}{2}) \right )\log \gamma\\
  & \leq \left (\frac{1}{d-1}\left ( \sum_{l=1}^{d-2}l+\frac{d-1}{2} \right
  )-\frac{d}{2}\right )\log \gamma \\
& = -\frac{1}{2}\log \gamma .
\end{align*}
 This proves
\[ \sup_{n\in\N} \left ( \log
  s_n(J) + \left (n\frac{(d-1)!}{2}\right )^{1/(d-1)}\log \gamma \right )\leq -\frac{1}{2}\log
\gamma.\]
To obtain equality we consider $s_{h_d(k)}(J)$ and again
apply Lemma (\ref{berg:lem:techlem}).

Finally, note that (\ref{berg:prop6:e4}) is simply a restatement of
(\ref{berg:prop6:e3}) and (\ref{berg:prop6:e3i}).
\end{proof}

\section{Singular values of arbitrary canonical identifications}
We shall now show how to extend Proposition \ref{ballasymp} to
identifications of
harmonic Bergman spaces on general open sets in $\R^d$. The main tool
is the following construction.

\begin{lem}\label{berg:lem3}
  Let $U,V,W\subset \R^d$ be open with $U\subset V\subset W$. Then
the operator
\[ T_U: h^2(V) \rightarrow h^2(W) \]
defined  by
\[ (T_Uf,g)_{h^2(W)}=\int_U f(x) \overline{g(x)}\,dx \]
is bounded  with $\|T_U\| \leq 1.$
\end{lem}
\begin{proof}
  Indeed
\[ \abs{\int_U f(x) \overline{g(x)}\,dx}^2\leq \left(\int_U
\abs{f}^2\right) \left( \int_U \abs{g}^2\right) \leq
\norm{f}_{h^2(V)}^2\, \norm{g}_{h^2(W)}^2 \]  implies that $T_U$
is well-defined and continuous with norm at most $1$.
\end{proof}

\begin{defn}
  Let $\{\Omega_n\}_{1\leq n\leq N}$  be
a finite collection of open subsets of $\R^d$. A collection
$\{\widetilde\Omega_n\}_{1\leq n\leq N}$ of mutually disjoint open
sets, with $\widetilde\Omega_n\subset \Omega_n$ for each $n$, is
called a \emph{disjointification} of $\{\Omega_n\}_{1\leq n\leq
N}$ if the symmetric difference $\left (\cup_{n=1}^N\Omega_n\right
)\triangle \left (\cup_{n=1}^N\widetilde\Omega_n\right)$ is a Lebesgue
null set.

\end{defn}
We note that if a collection $\{\Omega_n\}_{1\leq n\leq N}$ has the
property that the boundary of each $\Omega_n$ is a Lebesgue null set,
then a disjointification exists and can, for example, be obtained by setting

$$\widetilde \Omega_1  = \Omega_1, \quad
\widetilde\Omega_n={\rm int}\left( \Omega_n\setminus
\left(\bigcup_{i=1}^{n-1} \Omega_i\right)\right) \quad {\rm for}
\quad 2\le n\le N\,.
$$

The usefulness of the operator $T_U$
is due to the following result.

\begin{prop}\label{berg:prop:top}
  Let $\Omega\subset\R^d$ be open.
  Given open subsets $\Omega_1, \ldots, \Omega_N$ of
  $\Omega$, let
\[ J_n: h^2(\Omega) \to h^2(\Omega_n) \]
be the canonical identification. If
$\{\widetilde\Omega_n\}_{1\leq n\leq N}$ is a disjointification
of
$\{\Omega_n\}_{1\leq n\leq N}$,
then the
canonical identification
\[ J:h^2(\Omega)\to h^2\left(\bigcup_{n=1}^N\Omega_n\right) \]
can be written as
\[ J=\sum_{n=1}^NT_{\widetilde\Omega_n}J_n\,, \]
where
\[ T_{\widetilde\Omega_n}:h^2(\Omega_n)\rightarrow
h^2\left(\bigcup_{n=1}^N\Omega_n\right) \] is the operator defined
in Lemma \ref{berg:lem3}.
\end{prop}
\begin{proof}
  Let $f\in h^2(\Omega)$ and $g\in h^2(\bigcup_{n=1}^N\Omega_n)$. Then
\begin{align*}
  (\sum_{n=1}^NT_{\widetilde\Omega_n}J_nf,g)_{h^2\left(\bigcup_n\Omega_n\right)}
  &=\sum_{n=1}^N\int_{\widetilde\Omega_n}f(x)\overline{g(x)}\,dx\\
  &=\int_{\bigcup_n\Omega_n}f(x)\overline{g(x)}\,dx\\
  &=(Jf,g)_{h^2(\bigcup_n\Omega_n)}
\end{align*}
and the assertion follows.
\end{proof}

Before
proving the main result of this section we require some more
terminology.

\begin{defn}\label{cover}
  Let $\Omega_1,\Omega_2\subset
\R^d$ be open with $\Omega_2\cc \Omega_1$.
  Let $N\in\N$.
A finite collection
$B_1,\ldots,B_N$ of balls is called a
  {\it relative cover} of the pair
$(\Omega_1,\Omega_2)$ if the following two conditions hold:

  \begin{itemize}
  \item[(a)] $\Omega_2\subset \bigcup_{n=1}^NB_n\,$;
  \item[(b)] for each $1\le n\le N$, there exists
    $\gamma_n>1$ such that
    $\bigcup_{n=1}^N B_n(\gamma_n)\subset \Omega_1$.
 \end{itemize}
We call $N$ the {\it size} and $(\gamma_1,\ldots,\gamma_N)$ a {\it
  scaling} of the relative cover.

Given a relative cover $B_1,\ldots,B_N$ of $(\Omega_1,\Omega_2)$
with scaling $(\gamma_1,\ldots,\gamma_N)$, the vector
\[ \Gamma=(\log \gamma_1,\ldots,\log\gamma_N)\in\R_+^N\]
is called the {\it efficiency} of the relative cover. We define
$$\|\Gamma\| = \min_{1 \le j\le N}|\log \gamma_j|$$ and for $k
\in \N$, $$\|\Gamma\|_k =\left(\sum_{j=1}^N |\log
\gamma_j|^{-k}\right)^{-1/k}.$$
 \end{defn}

We note that, since $\Omega_2$ is relatively compact in
$\Omega_1$, there always
  exists a relative cover for $(\Omega_1,\Omega_2)$.

We are now able to prove the main result of this section.

\begin{theorem}
\label{embeddingtheorem}
  Let $d\geq 2$ and let $\Omega_1,\Omega_2\subset \R^d$ be open with $\Omega_2\cc
  \Omega_1$. Suppose that $\{B_n\}_{1\le n\le N}$ is a relative
  cover of $(\Omega_1,\Omega_2)$ of size $N$ with efficiency
$\Gamma$. Then the canonical
identification
\[ J:h^2(\Omega_1)\to h^2(\Omega_2) \]
satisfies
\[
J\in E(c,1/(d-1)), \quad \text{where}\quad
c=\left(\frac{(d-1)!}{2}\right )^{1/(d-1)}\norm{\Gamma}_{(d-1)},
\]
and
\begin{equation}
\label{jcguage} |J|_{c,1/(d-1)}\leq N \exp( -\norm{\Gamma}/2).
\end{equation}
\end{theorem}

\begin{proof}
Suppose that $\Gamma=(\log\gamma_1,\ldots,\log\gamma_N)$, where
$(\gamma_1,\ldots,\gamma_N)$ is a scaling of
$\{B_n\}_{1\le n\le N}$.
  Let $\{\widetilde\Omega_n\}_{1\leq n\leq N}$ be a disjointification
  of $\{B_n\}_{1\le n\le N}$, and let
\[ T_{\widetilde\Omega_n}:h^2(B_n)
\rightarrow h^2\left(\bigcup_{n=1}^NB_n\right)
\quad (1\leq n\leq N)\] denote the operator defined in Lemma
\ref{berg:lem3}.  Consider the following canonical identifications:
\[ \widetilde J_n:h^2(\Omega_1)\to h^2(B_n(\gamma_n))
\quad (1\leq n \leq N), \]
\[ J_n:h^2(B_n(\gamma_n))
\hookrightarrow h^2(B_n) \quad (1\leq n \leq N), \]
\[ \widetilde J:h^2
\left(\bigcup_{n=1}^NB_n\right)\to h^2(\Omega_2). \]
By Proposition \ref{berg:prop:top} we have
\begin{equation}
\label{bigfactorization}
 J=\sum_{n=1}^N\widetilde J T_{\widetilde\Omega_n}J_n\widetilde
J_n\,.
\end{equation}

Since $\|\widetilde J\|\leq 1$ and $\|\widetilde J_n\|\leq 1$,
while $\norm{T_{\widetilde\Omega_n}}{}\leq 1$ by Lemma
\ref{berg:lem3}, we conclude, by Propositions \ref{prop2} and
\ref{ballasymp}, that for $1 \le n \le N$,
\[ \widetilde J T_{\widetilde\Omega_n}J_n\widetilde J_n \in E(c_n,1/(d-1)),
\quad \text{with } c_n=\left(\frac{(d-1)!}{2}\right )^{1/(d-1)} \log\gamma_n
\]
and
\[
|\widetilde J T_{\widetilde\Omega_n}J_n\widetilde J_n|_{c_n,1/(d-1)}\leq
\exp( -(1/2)\log\gamma_n ).
 \]
 The assertion now follows from
Proposition~\ref{prop2}.
\end{proof}

\section{Bounded operators on spaces of harmonic functions}

We are now able to prove the main result which gives explicit upper bounds
for the eigenvalues of bounded operators on the space $h(\Omega)$ of
harmonic functions on an open set $\Omega\subset\R^d$. In order to
specify a topology on $h(\Omega)$ we define, for each $\Omega'\subset \R^d$
with $\Omega'\cc\Omega$, the following seminorm on $h(\Omega)$
\[ p_{\Omega'}(f):=\sqrt{ \int_{\Omega'}\abs{f(x)}^2 dx}\,. \]
If $\set{\Omega_n}_{n\in\N}$ is a collection of open subsets of $\R^d$
such that
\begin{itemize}
\item[(i)] $\Omega_n\cc\Omega_{n+1} \text{ for every $n\in\N$}$,
\item[(ii)] $\bigcup_{n\in\N}\Omega_n=\Omega$,
\end{itemize}
then $\{p_{\Omega_n}\}$ forms a directed system of seminorms which
turns $h(\Omega)$ into a Fr\'echet space, whose topology is equivalent 
to the topology of uniform
convergence on compact subsets of $\Omega$. Moreover, since each
canonical identification $h^2(\Omega_{n+1})\to h^2(\Omega_n)$ is
nuclear by Theorem~\ref{embeddingtheorem}, the space $h(\Omega)$ is
nuclear. 

A study of other kinds of harmonic function spaces can be
found in \cite{chu}.

Recall that a subset $S$ of a topological vector
space $E$ is {\it bounded} if for each neighbourhood $U$ of $0$,
we have $S \subset \alpha U$ for some $\alpha > 0$. A linear
operator $L:E\to E$ is {\it bounded} if it takes a neighbourhood
of
  zero into a bounded set. In order to formulate the main result, we
  require the following definition.
  \begin{defn}
    Let $\Omega, \Omega'\subset \R^d$ be open with $\Omega'\cc \Omega$. A
    linear operator $L:h(\Omega)\to h(\Omega)$ is called
    \textit{$\Omega'$-bounded} if for every $\Omega''\cc\Omega$ there
    is a positive constant $k$ such that
\[ p_{\Omega''}(Lf)\leq k p_{\Omega'}(f) \quad \text{ for every $f\in
    h(\Omega)$}\,. \]
 \end{defn}
Clearly, a linear operator $L:h(\Omega)\to h(\Omega)$ is bounded if
and only if it is $\Omega'$-bounded for some $\Omega'\cc\Omega$.

We shall now discuss some natural examples of bounded operators on
$h(\Omega)$ for plane domains $\Omega$. We identify the complex
plane $\Bbb C$ with ${\Bbb R}^2$. By a {\it conformal map} on
$\Omega$ we mean a holomorphic map $\varphi: \Omega \to
\Bbb C$ such that the derivative $\varphi'$ has no zero, in which
case, the differential
$$ d\varphi (z) : \R^2 \to \R^2$$
at each $z \in \Omega$ is a linear isomorphism since $\det
d\varphi(z) = |\varphi'(z)|^2 \neq 0$. Hence $\varphi$ is a local
diffeomorphism on $\Omega$. Let $\Pi =\{z \in {\Bbb R}^2: \Im\,z >0
\}$ be the upper half-plane in ${\Bbb R}^2$. We can define a
conformal mapping from $\Pi$ onto a proper region $\Omega_2 \cc
\Pi$. For instance, the conformal map
$$\psi : z \in U \mapsto \left(\frac{1+z}{1-z}\right)^2 \in \Pi$$
sends the semicircular disc $U \subset \Pi$, centred at $0$ with
radius $1$, onto $\Pi$. Note that $\psi$ is one-to-one on U. Hence
the conformal map $\varphi =\psi^{-1} + 2i$ sends $\Pi$ onto the
translation $\Omega_2 = U + 2i$ of $U$ with $\Omega_2 \cc \Pi$.
Another example is the Schwarz-Christoffel transformation $$ \varphi
(z) = \int_0^z \frac{dt}{\sqrt{(1-t^2)(1-t^2/4)}} \qquad (z \in
\Omega)\,,$$ which maps $\Pi$ conformally onto a rectangle in $\Pi$ (cf.
\cite[p. 231]{a}).

\begin{ex}\label{operator}
Let $\Omega \subset {\Bbb R}^2$ be a domain and let $\varphi :
\Omega \to \Omega'$ be a conformal bijection whose image
$\Omega'$ satisfies $\Omega' \cc \Omega$. Since harmonic functions
on plane domains are real parts of holomorphic functions, one can
define a composition operator $L_\varphi : h(\Omega)
\to h(\Omega)$  by
$$L_\varphi (f) = f \circ \varphi \qquad (f \in h(\Omega)).$$
Then $L_\varphi$ is $\Omega'$-bounded. To see this, let $\Omega''$ be
open with $\Omega'' \cc \Omega$. Then we have, via a change of variable,
\begin{eqnarray*}
p_{\Omega''}(L_\varphi f)^2 & = & \int_{\Omega''} |f\circ \varphi
(x)|^2\,dx
 =  \int_{\varphi(\Omega'')} |f (x)|^2\,|\det\, d \varphi^{-1}(x)|\,dx\\
& \leq & k\int_{\varphi(\Omega'')} |f (x)|^2 \,dx \\
& \leq & k p_{\Omega'}(f)^2\,,
\end{eqnarray*}
where $k=\sup_{x\in\varphi(\Omega'')}|\det d\varphi^{-1}(x)|<\infty$. 
Thus $L_\varphi$ is $\Omega'$-bounded. 

More generally, if $\varphi:\Omega\to \Omega'$ is any conformal map
between plane domains with $\Omega'\cc\Omega$, then a local change of
variables together with a compactness argument shows
that $L_\varphi$ is bounded in this case as well.  
\end{ex}

For open sets $\Omega$ in Euclidean space $\R^d$ of dimension
greater than $2$, one can construct bounded composition operators
$L_\varphi$ on $h(\Omega)$ analogous to the above example, but the
choice of $\varphi: \Omega \to \Omega$ is more delicate.
A smooth map $\varphi: \Omega \to \Omega$ for which the
composition operator $L_\varphi : h(\Omega) \to
h(\Omega)$ is well-defined is called a {\it harmonic morphism}.
Harmonic morphisms between Riemannian manifolds have been
characterized and widely studied. We refer to \cite{bw} for details
and examples.

 We are now ready to prove the main result.
\begin{theorem}
\label{maintheorem}
Let $\Omega$ be open in $\R^d$ $(d >1)$ and
let
   $L:h(\Omega)\to h(\Omega)$ be $\Omega'$-bounded for some
   $\Omega'\cc\Omega$.
If $\Omega''$ is open with
\[ \Omega'\cc\Omega''\cc\Omega \]
and such that $(\Omega'',\Omega')$ has a relative cover of size $N$
   and efficiency $\Gamma$, then
\[
\lam(L)\in {\cal E}(c,1/(d-1))\quad \text{with
$\abs{\lambda(L)}_{c,1/(d-1)}\leq KN\exp(-\|\Gamma\|/2)$,}
\]
where
\[ c=\frac{d-1}{d}\left (\frac{(d-1)!}{2}
  \right )^{1/(d-1)}\|\Gamma\|_{(d-1)} \]
and
\[ K=\sup\all{p_{\Omega''}(Lf)}{f\in h(\Omega),\,p_{\Omega'}(f)\leq 1}\,.\]
\end{theorem}
\begin{proof}
Define the following canonical identifications
\[ J_1:h(\Omega)\to h^2(\Omega''), \]
\[ J_2:h^2(\Omega'')\to h^2(\Omega'). \]
Clearly, $J_1$ and $J_2$ are continuous. Let
$\overline{J_2J_1h(\Omega)}$ be the closure of $J_2J_1h(\Omega)$
in the Hilbert space $h^2(\Omega')$ and let $P: h^2(\Omega')
\to \overline{J_2J_1h(\Omega)}$ be the natural
projection. Since $L$ is $\Omega'$-bounded,
the linear map $$f|_{\Omega'} \in
J_2J_1h(\Omega) \mapsto Lf \in h(\Omega) \qquad (f \in
h(\Omega))$$ is well defined and bounded, and therefore extends to
a bounded linear map
\[  \widetilde L : \overline{J_2J_1h(\Omega)} \to h(\Omega). \]
We now observe that $L$ admits the following factorisation
\[ L=\widetilde{L}PJ_2J_1. \]
By Pietsch's principle of related operators
(see \cite[Satz 1 and Satz
2]{pietsch}) it follows that
\[\lam(L) =  \lam(\widetilde{L}PJ_2J_1)=\lam(J_1\widetilde{L}P J_2)\,. \]
But since
$J_1\widetilde{L}P :h^2(\Omega')\to h^2(\Omega'')$ is
bounded with norm $K$ and since, by Theorem~\ref{embeddingtheorem},
we have $J_2\in E(c',1/(d-1))$ with
$|J_2|_{c',1/(d-1)}\leq N\exp(-\|\Gamma\|/2)$,
where
\[ c'=\left ( \frac{(d-1)!}{2} \right ) ^{1/(d-1)}\|\Gamma\|_{(d-1)}\,,\]
it follows by Proposition~\ref{prop2} that
$\lambda(J_1\widetilde{L}PJ_2)\in {\cal E}((d-1)c'/d,1/(d-1))$ with
$ |J_1\widetilde{L}PJ_2|_{(d-1)c'/d,1/(d-1)}\leq
KN\exp(-\|\Gamma\|/2)$. Thus $\lambda(L)$ has the desired properties.
\end{proof}
An immediate consequence of the previous theorem is the following
analogue of Grothendieck's Remarque 9 mentioned in the introduction.
\begin{coro} Let $\Omega$ be open in $\R^d$ $(d >1)$ and
let
   $L:h(\Omega)\to h(\Omega)$ be a bounded linear operator.
   Then $\lam(L)\in {\cal
  E}(1/(d-1))$.
\end{coro}

\begin{ex} As in the discussion before Example \ref{operator},
let $\varphi : \Pi \to \Omega_2$ be the one-to-one
conformal map from the upper half-plane $\Pi$ onto the semicircular
disc $\Omega_2$, centred at the point $2i$ with radius $1$. Consider
the composition operator $L_\varphi : h(\Pi) \to h(\Pi)$
defined in Example \ref{operator}. Let $\Omega_1$ be the open
disc $B_{\gamma,2i}$ centred at the point $2i$ with radius $1<
\gamma <2$. Then we have $\Omega_2 \cc \Omega_1 \cc \Pi$ while
$B_{2, 2i} \not\cc \Pi$. By Definition \ref{cover}, the singleton
$\{B_{1,2i}\}$ is a relative cover of $(\Omega_1, \Omega_2)$ with
{\it optimal} scaling $\gamma$, and efficiency $\Gamma =
\log\,\gamma$. We have $d-1=1$ for  $\Pi \subset \R^2$. Hence
$\|\Gamma\|= \log\,\gamma = \|\Gamma\|_1$ and $c =
\frac{1}{4}\log\,\gamma$, as in Theorem \ref{maintheorem} which
gives the following eigenvalue asymptotics
$$|\lambda_n (L_\varphi)| \leq K\exp(-(\log \gamma)/2)\exp(-(\log \gamma)n/4)$$
where
$$K = \sup_{z\in\varphi(B_{\gamma,2i})} \sqrt{|\det\,
  d\varphi^{-1}(z)|}
=\sup_{z\in \varphi(B_{\gamma,2i})}\abs{\frac{4-8i+z}{(1+2i-z)^3}}
\,.$$
In particular, we see that 
$$\lambda _n (L_\varphi) = O(\gamma^{-n/4})$$
for every $\gamma<2$. 
\end{ex}

\section{Acknowledgement}
We would like to thank an anonymous referee for useful comments and suggestions
that helped to
improve the presentation of this article.


\begin{thebibliography}{999}

\bibitem{a} L.V. Ahlfors, \textit{Complex Analysis}, McGraw-Hill,
New York, 1966.

\bibitem{ABR} S.~Axler, P.~Bourdon, and W.~Ramey,
    \textit{Harmonic Function Theory}, Springer-Verlag, New-York, 2001.

\bibitem{bw} P. Baird and J.C. Wood, 
\textit{Harmonic morphisms between Riemannian manifolds},
LMS Monograph \textbf{29}, Oxford Univ. Press, Oxford, 2003.

\bibitem{Ban} O.F.~Bandtlow, 
Resolvent estimates for operators belonging to exponential classes, 
\textit{Integr.\  Equ.\  Oper.\  Theory} \textbf{61} (2008) 21--43. 


\bibitem{BJ} O.F.~Bandtlow and O.~Jenkinson, Explicit eigenvalue 
estimates for transfer operators acting on spaces of 
holomorphic functions, \textit{Adv.\ Math.} \textbf{218} (2008)
902--925. 

\bibitem{chu} C-H. Chu, Harmonic function spaces on groups, \textit{J.
London Math.\ Soc.} \textbf{70}  (2004) 182-198.

\bibitem{Gro} A.~Grothendieck, Produits tensoriels
  topologiques et espaces nucl\'eaires,  {\it Mem. Amer. Math.
  Soc.}  \textbf{16}, 1955.

\bibitem{m} C.~M\"uller, \textit{Spherical harmonics}, Lecture Notes in
Mathematics \textbf{17}, Springer-Verlag, Berlin, 1966.

\bibitem{pietsch} A.~Pietsch, Zur Fredholmschen Theorie in
  lokalkonvexen R\"aumen, \textit{Studia Math.} \textbf{22} (1963)
  161--179.

\bibitem{rudin} W.~Rudin, \textit{Real and complex analysis}, McGraw-Hill,
New York, 1966.


\end{thebibliography}
\end{document}